\documentclass[11pt]{article}
\usepackage{amsmath}
\usepackage{amsthm}
\usepackage{amssymb}
\usepackage[latin1]{inputenc}
\usepackage[english]{babel}
\usepackage{hyperref}
\usepackage{multicol}
\usepackage{comment}
\usepackage{color}

\theoremstyle{plain}
\newtheorem{Theorem}{Theorem}[section]
\newtheorem{Remark}[Theorem]{Remark}
\newtheorem{Lemma}[Theorem]{Lemma}

\newtheorem{Corollary}[Theorem]{Corollary}
\newtheorem{Proposition}[Theorem]{Proposition}
\newtheorem{Definition}[Theorem]{Definition}

\def\R{\mathbb{R}}
\def\N{\mathbb{N}}
\def\C{\mathbb{C}}

\def\d{\partial}
\def\bd{\bar{\partial}}
\def\o{\omega}

\def\p{\varphi}
\def\e{\varepsilon}

\def\a{\alpha}
\def\b{\beta}
\def\bb{\bar{\beta}}
\def\de{\delta}
\def\cp{\chi \varphi}
\def\bcp{\bar{\chi} \varphi}

\usepackage{amsthm,amsmath,amssymb,amscd}
\usepackage{mathrsfs}
\usepackage{lscape}
\usepackage{graphics} 
\usepackage[all,cmtip]{xy}
\usepackage{graphicx}
\usepackage{subfigure}
\usepackage{url}	
\usepackage{hyperref}
\usepackage{color}
\usepackage[usenames,dvipsnames]{xcolor}
\usepackage{verbatim}
\usepackage{fancyhdr}
\usepackage{geometry}
\usepackage{cancel}
\usepackage[normalem]{ulem}
\usepackage{mathtools}
\usepackage[english]{babel}
\mathtoolsset{showonlyrefs,showmanualtags}
\numberwithin{equation}{section}

\theoremstyle{plain}

\theoremstyle{definition}

\theoremstyle{remark}

\theoremstyle{definition} 

\theoremstyle{remark}

\newcommand{\st}{\left(}
\newcommand{\dt}{\right)}

\newcommand{\NN}{\mathbb{N}}

\newcommand{\K}{K\"{a}hler}

\begin{document}
\title{\bf On the ADM mass of K\"ahler scalar flat ALE metrics}

\date{\today}

\author{Claudio Arezzo\footnote{International Centre for Theoretical Physics ICTP,\, arezzo@ictp.it} \,and Karen Corrales \footnote{Facultad de Matem\'aticas, Pontificia Universidad Católica de Chile,\, kacorrales@mat.uc.cl}}
\date{\small}

\maketitle \vspace{-1cm}

\begin{abstract}
	\noindent In this paper we study the behaviour of scalar flat K\"ahler ALE spaces and their ADM mass under blow ups. In particular we prove that by blowing up sufficiently many points at sufficiently big mutual distance one can produce scalar flat metrics with arbitrarily large ADM mass. A general machinery for producing scalar flat non Ricci flat ALE spaces of zero ADM mass is also presented, using and integrating previous work by Rollin-Singer (\cite{RS}) and Hein-LeBrun (\cite{HL}).
\end{abstract}


\section{Introduction}\label{intro}

\noindent The existence and geometric properties of Ricci and scalar flat metrics on ALE K\"ahler spaces  with a given structure group at infinity have been subject of intense research in past few decades (see e.g. \cite{C}, \cite{L}, \cite{K}, \cite{J}, \cite{CS}, \cite{HV}, \cite{LV}). In this paper we study to which extent the machinery developed 
in \cite{AP} for constructing constant scalar curvature K\"ahler metrics on blow ups can be used and extended in this setting.
Given the fact that such spaces do not carry holomorphic vector fields (which are the key obstruction for this theory) it is not hard to get the existence of scalar-flat metrics on blow ups of scalar flat ALE spaces (Theorem \ref{cscKALE}). While not the main focus of this paper, we point out that the new ALE scalar flat spaces 
produced in Theorem \ref{cscKALE} can in turn be used as new models for the gluing theory, hence getting, even starting from  compact cscK manifolds, infinitely may new families of cscK manifolds of different topological types. 

\noindent The new interesting problem studied in this paper is to estimate the maximal size of an exceptional divisor that can be inserted in the original space. It is known that in general this question can be very subtle on compact manifolds, carrying implicitily the difficulties hidden in the notion of $K$-stability and in the Tian-Yau-Donaldson Conjecture (see for example Section $7.2$ point (2) in \cite{D} and \cite{S2}).  On top of its clear geometric interest, the main motivation for this question comes from the recent and beautiful work by Hein-LeBrun (\cite{HL}) connecting various geometric and topological quantities of ALE spaces to  their ADM mass, another natural invariant arising from General Relativity recalled in Section 2.2 (\cite{ADM}).

\noindent Hein-LeBrun's topological formula for the ADM mass (Theorem \ref{admeq}) immediately implies that the effect of blowing up one point on the ADM mass is to increase it by an amount proportional to the volume of the exceptional divisor that one can insert, while keeping the scalar curvature zero, hence pointing again to the problem mentioned above.

\noindent The main result of this paper (Theorem \ref{epsilon}) guarantees that this quantity cannot decrease when performing a second blow up at a point sufficiently far away from the first one.  Iterating this scheme for multiple blow ups at sufficiently many points at sufficiently big mutual distance, we get the main applications in terms of ADM mass.
 
\begin{Theorem}\label{mainth1}
Let  $(X,{\bf{\Gamma}}, h, \omega)$ be an ALE scalar flat \K\ manifold with structure group $\bf{\Gamma}$ and ADM mass $e \in \R$. Then for every $\tilde{e} \geq e$ there exists an ALE scalar flat \K\ manifold with structure group $\bf{\Gamma}$ with ADM mass $\tilde{e}$. Such manifold can be obtained from $X$ by blowing up sufficiently many points at sufficiently big mutual distance.
\end{Theorem}

\noindent The above theorem indicates that, while $\Gamma$ does not determine the ADM mass of an ALE scalar flat \K\ manifold with structure group $\bf{\Gamma}$ (something easy to observe even for trivial $\Gamma$), the interesting quantity that can be algebraically determined by the structure group is 

$$inf\{\mbox{ADM mass} (X,{\bf{\Gamma}}, h, \omega) \, | \, (X,{\bf{\Gamma}}, h, \omega) {\mbox{ ALE scalar flat \K\ with structure group }} \, \bf{\Gamma} \}.$$

\noindent Theorem \ref{mainth1} gives also a general framework to produce scalar flat, non Ricci flat, ALE spaces with vanishing ADM mass. Specific examples were known thanks to Rollin-Singer (\cite{RS}, Section 6.7) and Hein-LeBrun (\cite{HL}, Theorem 4.7), but our result gives a general machinery to produce them.

\begin{Corollary}\label{maincor}
Let  $(X,{\bf{\Gamma}}, h, \omega)$ be an ALE scalar flat \K\ manifold with structure group $\bf{\Gamma}$ and ADM mass $e <0$. Then there exists an ALE scalar flat \K\ manifold with structure group $\bf{\Gamma}$ with vanishing ADM mass. Such manifold can be obtained from $X$ by blowing up sufficiently many points at sufficiently big mutual distance.
\end{Corollary}

\section{Preliminaries}

\subsection{The \K\ potential of a scalar flat ALE \K\ resolution.}

We start by recalling the concept of Asymptotically Locally Euclidean (ALE for short) K\"ahler manifold with structure group $\bf{\Gamma}$. Let ${\bf{\Gamma}} = (\Gamma_j), j=1...,k$ be a collection of finite subgroups of the unitary group 
$U(m)$ each acting freely away from the origin. We say that a complete noncompact \K\ manifold $(X,{\bf{\Gamma}}, h, \omega)$ of complex dimension $m$, where $h$ is the \K\ metric and $\omega$ is the \K\ form, is an ALE \K\ manifold with structure group $\bf{\Gamma}$ if 
\begin{itemize}
\item
there exist a compact subset $K$ such that the $j$-th connected component of $X\setminus K$ is biholomorphic via a map $\pi_j$ to
$(\C^m\setminus B)/\Gamma_j$, where $B$ is a closed ball with positive radius $R>0$, and
\item on each component in standard Euclidean coordinates the metric $\pi_*h$ satisfies 
\begin{equation}
\left|  \frac{\partial^\alpha}{\partial x^\alpha} \left(  \big (\pi_{*}h)_{i \bar{j}}  \, - \, \frac{1}{2} \, \delta_{i\bar{j}}  \right) \right| \,\, = \,\,  \mathcal{O}\st |x|^{-\tau - |\alpha|}\dt\,, 
\end{equation}
for some $\tau>0$ and every multindex $\alpha \in \NN^m$. 
\end{itemize}

\noindent Note that we are identifying the complex structure outside a compact subset with the standard one. This puts a restriction (widely studied) on the class of spaces allowed since in general its complex structure could be only asymptotic to the standard one and in fact the complex structure could not even admit holomorphic coordinates at infinity as shown for example by Honda (\cite{H}) also in the scalar flat case.

\vspace{11pt}

\noindent {\it{While the definition allows for an ALE space to have multiple components, a remarkable Theorem by Hein-LeBrun (Propositions $1.5$ and $4.2$ in \cite{HL}) shows that only the case $k=1$ can actually occur. We can then assume it thoughtout the paper.}}

\vspace{11pt}

\noindent The behaviour of the \K\ potential of a scalar flat ALE \K\ metric is quite well understood  (see e.g.\cite{AP} Section $7$ and \cite{advm} for further details)

\begin{Proposition}\label{asintpsieta2}
Let $(X,\Gamma, h, \omega)$ be a scalar flat ALE \K\ manifold with structure group $\Gamma$. Then for $R>0$ large enough, we have that on $X\setminus \pi^{-1}(B_R)$ the K\"ahler form can be written as
\begin{itemize}
\item
if $dim(X) = m >2$, then, for some real constant $e_{X}$,
\begin{equation}
\omega \,\, = \,\, i\partial\overline{\partial} \st \,  \frac{|x|^2}{2} \, + \, e_{X} \, |x|^{4-2m}   + \, \psi_{\eta}\st x \dt \dt \,,  \qquad \mbox{with} \qquad \psi_\eta \, = \,  \mathcal{O}(|x|^{3-2m}) \, ,
\end{equation}
\item
if $dim(X) = 2$, then for some real constant $e_{X}$,
\begin{equation}
\omega \,\, = \,\, i\partial\overline{\partial} \st \,  \frac{|x|^2}{2} \, + \, e_{X} \, log|x|  + \, \psi_{\eta}\st x \dt \dt \,,  \qquad \mbox{with} \qquad \psi_\eta \, = \,  \mathcal{O}(|x|^{-1}) \, ,
\end{equation}
\end{itemize}
\end{Proposition}

\subsection{The ADM mass of a \K\ potential of an ALE \K\ manifold with structure group $\Gamma$.}

The key quantity studied in this paper is the celebrated ADM mass introduced by Arnowitt-Deser-Misner (\cite{ADM}), whose 
geometrical nature has been clarified by Bartnik (\cite{B}) and Chru\'sciel (\cite{C}).

\noindent In the context on ALE \K\ manifolds it has been the center of the beautiful paper by Hein-LeBrun (\cite{HL}) which is the original motivation of our work. 

\noindent Following Hein-LeBrun's normalization conventions are chosen so that, at a given end, the mass of an ALE manifold is given by
$$
m(X,h) = \lim\limits_{R\rightarrow \infty} \frac{{\bf{\Gamma}}(\frac{m}{2})}{4(m-1)\pi ^{\frac{m}{2}}}
\int_{S_R/\Gamma} [h_{kl,k} - h_{kk,l}] {\bf{n}}^lda_E $$
where $S_R$ is the Euclidean coordinate sphere of radius $R$, $da_E$ is the $(m - 1)$-dimensional volume form induced on this sphere by the Euclidean metric, and $\bf{n}$ is the outward-pointing Euclidean unit normal vector. 

\noindent Let us also denote by $\clubsuit \colon H^2 \rightarrow H^2_c(X)$ the inverse of the isomorphism induced by the standard inclusion of compactly supported forms into all differential forms.

\noindent Restricting our attention to the scalar flat case, their key result is the following equivalence:

\begin{Theorem}\label{admeq}

Let $(X,\Gamma, h, \omega)$ be a scalar flat ALE \K\ manifold with structure group $\Gamma$. Then 

$$ m(X,h) = - \frac{\langle \clubsuit(c_1(X), \omega^{m-1}\rangle}{(2m-1)\pi^{m-1}} = e_X \,\, .$$

\end{Theorem}

\noindent The first equality is a special case of Theorem C in \cite{HL}, while the equality $m(X,h) = e_X$ is a classical observation due to LeBrun (\cite{L}) and Rollin-Singer (\cite{RS}, Section 6.8) in complex dimension $2$, whose proof extends readily to any dimension.

\section{Blowing up points}

In order to fix notations needed in the sequel,  
for every point $p\in X$ we denote by $(z_1,\dots z_m)$ normal coordinates  centred at $p$, so that near $p$ the metric $\o$ is of the form
\begin{equation}
\label{omega}
\o=\o_E+i\d\bd \p_1(z),\quad \p_1(z)=\mathcal{O}(|z|^{4}),
\end{equation}
 specifically, $\p_1(z)\leq c|z|^4,$ where $c$ depends on the curvature of $\omega$ at $p$.

\noindent It is also convenient to isolate, among all ALE scalar flat \K\ manifolds, $(Bl_0\C^m, \{Id\}, h_{BS}, \eta_{BS})$, where $h_{BS}$ denotes the well-know Burns-Simanca metric, normalized so to have volume of the exceptional divisor $E$ equal to $1$.
Its asymptotic expansion at infinity can be written as
\begin{equation}
\label{eta}
\begin{aligned}
\eta_{BS}&=\o_E+i\d\bd \p_2(w),\quad \p_2(w)=\mathcal{O}(|w|^{4-2m}),\quad \mbox{for}\; m>2\\
\\
\eta_{BS}&=\o_E+i\d\bd \p_2(w),\quad \p_2(w)=\log(|w|^{2}),\quad \mbox{for}\; m=2
\end{aligned}
\end{equation}
in terms of the standard coordinates on $Bl_0\C^m\setminus E\cong \C^m\setminus \{0\}$.


\noindent On the other hand, \cite[Proposition 5.4]{AP} implies that there is no nonzero holomorphic vector field vanishing somewhere on $X$ and then it is easy to observe that the proof of  \cite[Theorem 1.1]{AP} readily extends to the ALE case. More precisely we have

\begin{Theorem}
	\label{cscKALE}
Let $(X,\Gamma, h, \omega)$ be a scalar flat ALE \K\ manifold with structure group $\Gamma$. Given finitely many smooth points $p_1,\dots p_k \in X$, there exists $\varepsilon_0(X, p_1, \dots, p_k) := \varepsilon_0 >0$ such that the blow-up of $X$ at  $p_1,\dots p_k$ carries zero scalar curvature K\"ahler forms 
$$\o_\e\in \pi^*[\o]-\e^2(PD[E_1]+\dots+PD[E_k]),$$
where the $PD[E_j]$ are the Poincar\'e duals of the $(2m-2)$-homology classes of the exceptional divisors of the blow-up at $p_j$ and $\e\in(0, \e_0)$.
\end{Theorem}

\noindent What is crucial for the scope of this paper is to make, as explicit as possible, the dependence of $\e_0$ on all the geometric data built in this theorem, such as the position of the points $p_j$ and the geometry of $X$ and of $Bl_0\C^m$.

\noindent Let us then recall that in order to construct the scalar flat metric on $Bl_{p_1,\dots p_k}X$, the idea is to replace the metric $\o$ on a small neighbourhood of $p_i$ with a suitably
scaled down copy of $\eta_{BS}$· 

\noindent Precisely, for a fixed $\e>0$ and $r_\e=\e^{m-1/m}$, we will glue $\o$ with $\e^2\eta_{BS}$ on the annulus $B_{2r_\e}\setminus B_{r_\e}$ around each $p_i$, under the change of variables $z=\e w.$

\noindent Following \eqref{omega} and \eqref{eta}, the metric on $X\setminus\{p_1\dots, p_k\}$, which naturally extends to a metric on $Bl_{p_1,\dots p_k}X$, is defined by
\begin{equation}
\label{omegae}
\tilde{\o}_\e=\begin{cases}
\o \quad \mbox{on}\quad X\setminus{\cup_{i=1}^k B_{2r_\e}(p_i)}\\
\o_E+i\d\bd(\gamma_1(z)\p_1(z) + \e^2(1-\gamma_1(z))\p_2(\e^{-1}z))\quad \mbox{on}\quad B_{2r_\e}(p_i)\setminus B_{r_\e}(p_i)\\
\e^2\eta_{BS} \quad \mbox{on}\quad B_{r_\e}(p_i)\setminus \{p_i\}
\end{cases}
\end{equation}

where $\gamma_1$ is a suitable cutoff function.

\medskip

\noindent Once this pre-glued metric is built the problem is to find a function $\p$ such that the metric $\o_\e=\tilde{\o}_\e+i\d\bd\p$ has zero scalar curvature. In order to solve that equation, the most important step is to show that
the linear operator $L_{\tilde{\o}_\e}$ associated to the scalar curvature is invertible and to obtain bounds on the norm of the inverse in suitable Banach spaces: weighted H\"older spaces, which now have to take into ``weighted" account the behaviour of function near the point $p_j$, with weight $\delta$, and at infinity, with weight $\delta_{\infty}$.
\vspace{3mm}

\noindent {\it{The weighted spaces where to carry out the proof of Theorem \ref{cscKALE} are the same one as those used in \cite{AP} and \cite{S} with the extra weight $\delta_{\infty}$. }}

\vspace{3mm}

\noindent If we denote $X^*:=X\setminus\{p_1\dots, p_k\}$, $X_r:=X\setminus{\cup_{i=1}^k B_{r}(p_i)}$ and $N_{p_i}$ by the identification of $B_{r}(p_i)\setminus \{p_i\}$ with $Bl_0\C^n$, then we define

\begin{Definition}
	Given $l\in\N$, $\a\in(0,1)$, $\de, \de_\infty\in (4-2m,0)$, we define
\begin{enumerate}
\item  The weighted space $C^{l,\a}_{\de,\de_\infty}(X^*)$ as the space of functions $\p\in C^{l,\a}_{loc}(X^*)$ for which the norm
$$\|\p\|_{C^{l,\a}_{\de,\de_\infty}(X^*)}:=\|\p\|_{C^{l,\a}_{\de_\infty}(X_{r_0})}+\max_{i=1...k} \left\{\sup _{r\in(0,2r_0)}\left(r^{-\de}\|\p(rz)\|_{C^{l,\a}(B_{2}(p_i)\setminus B_{1}(p_i))}\right)\right\}$$
is finite.
\item  The weighted space $C^{l,\a}_{\de}(N_{p_i})$ is defined to be the space of functions $\p\in C^{l,\a}_{loc}(N_{p_i})$ for which the norm
$$\|\p\|_{C^{l,\a}_{\de}(N_{p_i})}:=\|\p\|_{C^{l,\a}(\{|w|<2R_0\})}+\sup _{R\in[R_0,\infty)}\left(R^{-\de}\|\p(Rw)\|_{C^{l,\a}(B_{2}\setminus B_{1})}\right)$$
is finite.
\item The weighted space $C^{l,\a}_{\de,\de_\infty}(Bl_{p_1,\dots p_k}X)$ is defined for the gluing between the previous two, i.e. the space of functions $\p\in C^{l,\a}_{loc}(Bl_{p_1,\dots p_k}X)$ for which the norm
\begin{equation}
\begin{aligned}
\|\p\|_{C^{l,\a}_{\de,\de_\infty}(Bl_{p_1,\dots p_k}X)} &= \|\p\|_{C^{l,\a}_{\de_\infty}(X_{2r_\e})}+\|\p\|_{C^{l,\a}_{\de}(N_{p_i})} \\ 
& \;\;+  \max_{i=1...k} \left\{\sup _{r\in(r_\e,2r_\e)}\left(r^{-\de}\|\p(rz)\|_{C^{l,\a}(B_{2}(p_i)\setminus B_{1}(p_i))}\right)\right\}
\end{aligned}
\end{equation}
is finite.
\end{enumerate}
\end{Definition}

Throughout this paper we will denote $X_{p_1,\dots p_k}:=X\setminus{\cup_{i=1}^k B_{2r_\e}(p_i)}$ and $A_{p_i}:=B_{2r_\e}(p_i)\setminus B_{r_\e}(p_i)$, then
\begin{equation}
\label{norm}
\|\p\|_{C^{l,\a}_{\de,\de_\infty}(Bl_{p_1,\dots p_k}X)} = \|\p\|_{C^{l,\a}_{\de_\infty}(X_{p_1,\dots p_k})}+\|\p\|_{C^{l,\a}_{\de}(N_{p_i})} +\max_{i=1...k} \left\{\|\p\|_{C^{l,\a}_{\de}(A_{p_i})}\right\}.
\end{equation}


\begin{Remark}
Note that the norm $\|\p\|_{C^{l,\a}_{\de_\infty}(X_{p_1,\dots p_k})}$ uses to the metric $\o$; $\|\p\|_{C^{l,\a}_{\de}(A_{p_i})}$  the pre-glued metric $\tilde{\o}_\e$ and $\|\p\|_{C^{l,\a}_{\de}(N_{p_i})}$ uses the metric $\e^2\eta_{BS}$ for each $i$.
\end{Remark}

\begin{Remark}
It is important to remark that the proof of Theorem \ref{cscKALE} depends on whether $m=2$ or $m>2$. However, when $m= 2$ we can still work weight spaces with
$\delta \in(-1, 0)$, but we lose the uniform control of the inverse operator and other problems appear. One way to overcome these problems is to construct a better gluing metric than the one defined by \eqref{omegae}. If we consider

$$\hat{\omega}_\e=\begin{cases}
	\o \quad \mbox{on}\quad X\setminus{\cup_{i=1}^k B_{2r_\e}(p_i)}\\
	\o_E+i\d\bd(\gamma_1(z)\p_1(z) + \e^2(1-\gamma_1(z))\log|\e^{-1}z|)\quad \mbox{on}\quad B_{2r_\e}(p_i)\setminus B_{r_\e}(p_i)\\
	\e^2\eta_{BS} \quad \mbox{on}\quad B_{r_\e}(p_i)\setminus \{p_i\}
\end{cases}
$$

\noindent then we define $\tilde{\o}_\e=\hat{\omega}_\e+\e^2i\d\bd\left(\gamma_{1}(z)(\Gamma(z)-\log \varepsilon)\right)$, where $ \Gamma$ is a solution asymptotic to $\log |z|$ near $p_i$ of $L_{\omega}(\Gamma) = c$, where $c$ is a constant and $L_{\omega}$ is the linearized scalar curvature operato. For more details, see \cite[Section 8.4]{S}.

\noindent Therefore, the new analysis carried out in this paper does depend on $m$ we give the details only in the case $m>2$.
\end{Remark}

\noindent All proofs about the invertibility of the linearized  scalar curvature, $L_{\o_\e}$, contained in \cite{AP} and \cite{S}, go through verbatim also in this context so we do not reproduce them. However, we need to estimate the dependence of norms of operators, both linear and non-linear, on the geometric data of the construction. This is what we achieve in the next section.

\section{Dependence on the position of the points}

\noindent The goal in this section is to study the behaviour of $\e_0$ on Theorem \ref{cscKALE}, for  the blow-up at one point $p$ with the one for the blow-up at two points $p, q$, to prove that $\e_0(p,q)\geq \e_0(p)$ as $q$ goes to infinity.

\noindent Let $(X,\Gamma, h, \omega)$ be a scalar flat ALE \K\ manifold with structure group $\Gamma$ and pick two points $p,q\in X$. In the sequel we will consider $p\in X$ fixed and study the behaviour of all appearing quantities with $p$ fixed and moving $q$. Theorem \ref{cscKALE} implies that 
\begin{itemize}
	\item there exists $\varepsilon_0(p) >0$ such that for every $\e\in(0, \e_0(p))$ the blow-up of $X$ at  $p$ carries  scalar flat K\"ahler forms 
	$$\o_\e\in \pi^*[\o]-\e^2aPD[E].$$
\item there exists $\varepsilon_0(p,q) >0$ such that for every $\e\in(0, \e_0(p,q))$ the blow-up of $X$ at  $p, q$ carries constant scalar curvature K\"ahler forms 
	$$\o_\e\in \pi^*[\o]-\e^2a_pPD[E_p]-\e^2a_qPD[E_q].$$
\end{itemize}

\noindent Note that with a slight abuse of notation we use $\o_{\e}$ for the pre-glued metrics $\tilde{\o}_{\e}$ and the 
scalar flat metrics $\o_{\e}$ in both cases for $Bl_pX$ and $Bl_{p,q}X$ since there is no possibility of confusion.

\noindent Thus, our main estimate is

\begin{Theorem}
\label{epsilon}
Let $(X,\Gamma, h, \omega)$ be a scalar flat ALE \K\ manifold with structure group $\Gamma$ and $p,q\in X$. If $\e_0(p)$ and $\e_0(p,q)$ are the constants where Theorem \ref{cscKALE} holds then
$$\e_0(p,q)\to \e_0\geq \e_0(p),\quad \mbox{as } \;r(q) \to \infty$$
where $r(q)$ is the distance between $q$ and a compact subset $K\subset X$ containing $p$.
\end{Theorem}

\noindent In this case and following the metric defined in \eqref{omegae} we have
$$\tilde{\o}_\e=\begin{cases}
\o \quad \mbox{on}\quad X_{p,q}\\
\o_E+i\d\bd\p_{p,\e}(z)\quad \mbox{on}\quad A_{p}\\
\o_E+i\d\bd\p_{q,\e}(z)\quad \mbox{on}\quad A_{q}\\
\e^2\eta_{BS} \quad \mbox{on}\quad N_p\\
\e^2\eta_{BS} \quad \mbox{on}\quad N_q\\
\end{cases}$$
where $\p_{p,\e}(z)$ and $\p_{q,\e}(z)$ are the gluing potentials between $\p_1$ and $\p_2$, near $p$ and near $q$, respectively.

\begin{Remark}
	\label{remarkpotential}
It is important to note that \eqref{omega} and \eqref{eta} imply that there exist constants $c_p, c_q$ and $c$ such that
\begin{equation}
\label{potentialp}
\begin{split}
\|\p_{p,\e}(z)\|_{C^{4,\a}_{\de}(A_{p})}&\leq \|\p_{1}(z)\|_{C^{4,\a}_{\de}(A_{p})}+\e^2\|\p_{2}(\e^{-1}z)\|_{C^{4,\a}_{\de}(A_{p})}\leq (c_p+c)r_\e^{4-\de},\\
\|\p_{q,\e}(z)\|_{C^{4,\a}_{\de}(A_{q})}&\leq \|\p_{1}(z)\|_{C^{4,\a}_{\de}(A_{q})}+\e^2\|\p_{2}(\e^{-1}z)\|_{C^{4,\a}_{\de}(A_{q})}\leq (c_q+c)r_\e^{4-\de},
\end{split}
\end{equation}
where $c_p$ and $c_q$ depends on the curvature of $\o$ around $p$ and $q$, and $c$ does not depend on the point.  

\noindent The simple but crucial observation, repeatedly used in the sequel, is that since the metric $\o$ is ALE then $c_q\leq c_p$ as $q$ goes to infinity. Therefore,
\begin{equation}
	\label{potentialq}
	\|\p_{q,\e}(z)\|_{C^{4,\a}_{\de}(A_{q})}\leq (c_p+c)r_\e^{4-\de}.
\end{equation}
\end{Remark}


\noindent The existence of the constant $\e_0$ in the theory developed in \cite{AP} and \cite{AP2} comes from the estimates on the scalar curvature, its linear operator and the quadratic part to use the Fixed point Theorem. Denoting by $S(\o_\e)$ as the scalar curvature, $L_{\o_\e}$ its linear part, $Q_{\o_{\e}}$ its quadratic part and $\mathcal{N}$ as the operator
$$\mathcal{N}:C^{4,\a}_{\de,\de_\infty}(Bl_{p,q}X)\to C^{4,\a}_{\de,\de_\infty}(Bl_{p,q}X),\quad \mathcal{N}(\p)=L_{\o_\e}^{-1}(-S(\o_\e)-Q_{\o_\e}(\p)),$$ the key Proposition to prove the existence of cscK metrics is the following

\begin{Proposition}{see e.g. \cite[Proposition 8.20]{S}}
\label{fixpoint}
For some constant $c_1$, let
$$\mathcal{U}:=\{\p\in C^{4,\a}_{\de,\de_\infty}(Bl_{p,q}X) : \|\p\|_{C^{4,\a}_{\de,\de_\infty}(Bl_{p,q}X)}\leq c_1r_\e^{2-\de}\}.$$
Then, there exists $\e_0(p,q)$ such that for every $\e<\e_0$, $\mathcal{N}$ is a contraction on $\mathcal{U}$ and $\mathcal{N}(\mathcal{U})\subset \mathcal{U}.$ In particular, $\mathcal{N}$ has a fixed point, which gives a cscK metric on $Bl_{p,q}X$ in the K\"ahler class $\pi^*[\o]-\e^2a_pPD[E_p]-\e^2a_qPD[E_q]$.
\end{Proposition}

\noindent The key estimates to prove the above Proposition are the following 
\begin{description}
	\item [(i)] there exists a constant $c_1$ such that if $\|\p\|_{C^{4,\a}_{2,\de_\infty}(Bl_{p,q}X)}\leq c_1$ then 
	$$\|\mathcal{N}(\p)-\mathcal{N}(0)\|_{C^{4,\a}_{\de,\de_\infty}(Bl_{p,q}X)}\leq \frac{1}{2}\|\p\|_{C^{4,\a}_{\de,\de_\infty}(Bl_{p,q}X)}.$$
	\item[(ii)] for the same constant $c_1$, $\|\mathcal{N}(0)\|_{C^{4,\a}_{\de,\de_\infty}(Bl_{p,q}X)}\leq \frac{1}{2}c_1r_\e^{2-\de}$.
\end{description}
 
 \begin{Remark}
 	\label{fixpointKC}
 		Note that
 	\begin{equation*}
 	\|\mathcal{N}(\p)-\mathcal{N}(0)\|_{C^{4,\a}_{\de,\de_\infty}(Bl_{p,q}X)}=\|L_{\o_\e}^{-1}(Q_{\o_\e}(0)-Q_{\o_\e}(\p))\|_{C^{4,\a}_{\de,\de_\infty}}\leq K_{p,q}\|Q_{\o_\e}(\p))\|_{C^{0,\a}_{\de-4,\de_\infty-4}}\leq K_{p,q}c_1\|\p\|_{C^{4,\a}_{\de,\de_\infty}}
 	\end{equation*}
 	\begin{equation*}
 	\|\mathcal{N}(0)\|_{C^{4,\a}_{\de,\de_\infty}(Bl_{p,q}X)}=\|L_{\o_\e}^{-1}(-S({\o_\e}))\|_{C^{4,\a}_{\de,\de_\infty}(Bl_{p,q}X)}\leq K_{p,q}\|S(\o_\e))\|_{C^{0,\a}_{\de-4}(Bl_{p,q}X)}\leq K_{p,q}C_{p,q}r_\e^{4-\de}.
 	\end{equation*}
 	Therefore, {\rm\bfseries{(i)}} follows once $c_1$ is chosen less than $1/2K_{p,q}$ and $\e_0(p,q)$ is chosen such that for every $\e<\e_0$
 $$2K_{p,q}C_{p,q}r_\e^2<c_1, \quad r_\e=\e^{\frac{m-1}{m}}.$$
 \end{Remark}

\noindent Thus, the next two subsections will be dedicated to understand the definition and construction of $C_{p,q}$ and $K_{p,q}$. Next, they will be compared with the same constants $C_p$ and $K_p$ (in the case of $Bl_pX$) when $q$ goes to infinity.

\subsection{Comparison between $C_{p,q}$ with $C_p$}
In this part we will to compare the estimates on the scalar curvature for the preglued metrics for  $Bl_{p,q}X$ and $Bl_pX.$
It is known that there exist $C_p$ and  $C_{p,q}$ defined by
\begin{eqnarray*}
C_{p,q}&:=&\inf\{C>0\mid \|S(\o_\e)\|_{C^{0,\a}_{\de-4}(Bl_{p,q}X)}\leq C_{p,q}r_\e^{4-\de}\},\\
C_p&:=&\inf\{C>0\mid \|S(\o_\e)\|_{C^{0,\a}_{\de-4}(Bl_{p}X)}\leq C_{p}r_\e^{4-\de}\}.
\end{eqnarray*}

\begin{Lemma}
	\label{C}
	There exists a compact subset $K$ which contains $p$ such that $C_{p,q}\leq C_p$ for every $q\in X\setminus K$.
\end{Lemma}

\begin{proof}
By definition of the weighted norm, we have
\begin{equation}
\begin{split}
\label{S1}
\|S(\o_\e)\|_{C^{0,\a}_{\de-4}(Bl_{p}X)}&=\|S(\o_E+i\d\bd\p_{p,\e})\|_{C^{0,\a}_{\de-4}(A_p)}\\
\|S(\o_\e)\|_{C^{0,\a}_{\de-4}(Bl_{p,q}X)}&=\max\left\{\|S(\o_E+i\d\bd\p_{p,\e})\|_{C^{0,\a}_{\de-4}(A_p)},\|S(\o_E+i\d\bd\p_{q,\e})\|_{C^{0,\a}_{\de-4}(A_q)}\right\}\\
\end{split}
\end{equation}
Moreover, given that $S(\o)=(g_{\o})^{\a\bb}(R_{\o})_{\a\bb}$, then we compute 
\begin{eqnarray*}
	(g_{p,\e})_{j\bar{k}}&=&\de_{jk}+\frac{\d^2\p_{p,\e}}{\d z_j \d\bar{z}_k}(z)\\
	(g_{q,\e})_{j\bar{k}}&=&\de_{jk}+\frac{\d^2\p_{q,\e}}{\d z_j \d\bar{z}_k}(z)\\
	(R_{p,\e})_{\a\bb}&=& 2(g_{p,\e})^{j\bar{k}}\left[\left(\frac{\d^3 \p_{p,\e}}{\d z_\a\d z_j \d\bar{z}_k}\right)\left(\frac{\d^3\p_{p,\e}}{ \d \bar{z}_\b\d z_j \d\bar{z}_k}\right)-\left(\frac{\d^4\p_{p,\e}}{\d z_\a \d \bar{z}_\b\d z_j \d\bar{z}_k}\right)\right]\\
	(R_{q,\e})_{\a\bb}
	&=& 2(g_{q,\e})^{j\bar{k}}\left[\left(\frac{\d^3 \p_{q,\e}}{\d z_\a\d z_j \d\bar{z}_k}\right)\left(\frac{\d^3\p_{q,\e}}{ \d \bar{z}_\b\d z_j \d\bar{z}_k}\right)-\left(\frac{\d^4\p_{q,\e}}{\d z_\a \d \bar{z}_\b\d z_j \d\bar{z}_k}\right)\right]\\
\end{eqnarray*}

\noindent It is immediate to check that there exists a constant $c_1$  such that
$$\|\nabla^i\p_{p,\e}\|_{C^{4-i,\a}_{\de-i}}\leq c_1 \|\p_{p,\e}\|_{C^{4,\a}_{\de}}, \quad \|\nabla^i\p_{q,\e}\|_{C^{4-i,\a}_{\de-i}}\leq c_1 \|\p_{q,\e}\|_{C^{4,\a}_{\de}}, $$
where the constant $c_1$ is independent of the function and on the points.

\noindent Thus, there exist constants $C_1$, independent of $\p_{p,\e}$ and $\p_{q,\e}$, respectively, such that
$$\|S(\o_E+i\d\bd\p_{p,\e})\|_{C^{0,\a}_{\de-4}(A_p)}\leq C_1 \|\p_{p,\e}\|_{C^{4,\a}_{\de}(A_p)}$$
$$\|S(\o_E+i\d\bd\p_{q,\e})\|_{C^{0,\a}_{\de-4}(A_q)}\leq C_1 \|\p_{q,\e}\|_{C^{4,\a}_{\de}(A_q)}$$

\noindent Using Remark \ref{remarkpotential} taking a sufficiently big compact subset $K$ containing $p$, we obtain
\begin{equation}
\begin{split}
\|S(\o_\e)\|_{C^{0,\a}_{\de-4}(Bl_{p,q}X)}&= \max\left\{\|S(\o_E+i\d\bd\p_{p,\e})\|_{C^{0,\a}_{\de-4}(A_p)},
\|S(\o_E+i\d\bd\p_{q,\e})\|_{C^{0,\a}_{\de-4}(A_q)}\right\}\\
&\leq \max\left\{ C_1 \|\p_{p,\e}\|_{C^{4,\a}_{\de}(A_p)}, C_1 \|\p_{q,\e}\|_{C^{4,\a}_{\de}(A_q)}\right\}\\
&\leq C_1(c_p+c)r_\e^{4-\de}.
\end{split}
\end{equation}

for $q \in X\setminus K$.

\noindent Therefore, $C_{p,q}\leq C_1(c_p+c)=:C_p.$
\end{proof}

\subsection{Comparison between $K_{p,q}$ with $K_p$}

In this part we will to compare the bounds on the norm of the inverse linear operator associated to the scalar curvature in the construction of metrics on $Bl_{p,q}X$ and on $Bl_pX.$

\noindent It is known that for every $\p\in C^{4,\a}_{\de,\de_\infty}(Bl_{p,q}X) $ and $\psi\in C^{4,\a}_{\de,\de_\infty}(Bl_{p}X) $ there exist $K_{p,q}$ and $K_p$ defined by
\begin{eqnarray*}
	K_{p,q}&:=&\inf\{K>0\mid \|\p\|_{C^{4,\a}_{\de,\de_\infty}(Bl_{p,q}X)}\leq K\|L_{\o_\e}(\p)\|_{C^{0,\a}_{\de-4,\de_\infty-4}(Bl_{p,q}X)}  \},\\
	K_p&:=&\inf\{K>0\mid \|\psi\|_{C^{4,\a}_{\de,\de_\infty}(Bl_{p}X)}\leq K\|L_{\o_\e}(\psi)\|_{C^{0,\a}_{\de-4,\de_\infty-4}(Bl_{p}X)}\}.
\end{eqnarray*}

\begin{Lemma}
	\label{K}
$K_{p,q}$ goes to $K_0\leq K_p$ as $q$ goes to infinity.
\end{Lemma}

\begin{proof}
Let $\p\in C^{4,\a}_{\de,\de_\infty}(Bl_{p,q}X)$. We define the functions $\chi \p$ on $Bl_p(X)$ and $\bar{\chi}\p$ on $Bl_q X$ as follows
\begin{multicols}{2}
	$\cp=\begin{cases}
	\p(z)\quad \mbox{on } X_{p,q}\\
	\p(z)\quad \mbox{on } A_{p}\\
	\p(z)\quad \mbox{on } N_{p}\\
	\gamma(\frac{|z|}{r_\e})\p(z) \; \mbox{on } A_{q}\\
	0\quad \mbox{on } N_{q}
	\end{cases}$ 
	$\bcp=\begin{cases}
	0\quad \mbox{on } X_{p,q}\\
	0\quad \mbox{on } A_{p}\\
	0\quad \mbox{on } N_{p}\\
	\bar{\gamma}(\frac{|z|}{r_\e})\p(z) \; \mbox{on } A_{2q}\\
	\p(z)\quad \mbox{on } A_{q}\\
		\p(z)\quad \mbox{on } N_{q}
	\end{cases}$
\end{multicols} 

\noindent where $\gamma, \bar{\gamma}:\R\to [0,1]$ are smooth functions such that
\begin{multicols}{2}
	$\gamma(x)=\begin{cases}
1 \quad \mbox{if} \quad x\geq 2\\
0 \quad \mbox{if} \quad x\leq 1
\end{cases}$ and $\quad$
$\bar{\gamma}(x)=\begin{cases}
0 \quad \mbox{if} \quad x\geq 3\\
1 \quad \mbox{if} \quad x\leq 2
\end{cases}$
\end{multicols}

\noindent Note that it is possible to identify $\bcp$ around $p$, i.e. to compute the norm of the same function but with the metric near $p$. In this case, there exists a constant $K_1$, depending on $p$, such that
\begin{equation*}
\begin{split}
\|\bcp\|_{C^{4,\a}_{\de,\de_\infty}(Bl_{q}X)}&=\|\bcp\|_{C^{4,\a}_{\de_\infty}(X_q)}+\|\bcp\|_{C^{4,\a}_{\de}(A_{q})}+\|\bcp\|_{C^{4,\a}_{\de}(N_{q})} \quad \mbox{($\bcp$ has compact support)}\\
&\leq K_1(\|\bcp\|_{C^{4,\a}(X_p)}+\|\bcp\|_{C^{4,\a}_{\de}(A_{p})})+\|\bcp\|_{C^{4,\a}_{\de}(N_{p})}.
\end{split}
\end{equation*}

\noindent Since $$\|\p\|_{C^{4,\a}_{\de,\de_\infty}(Bl_{p,q}X)}=\|\p\|_{C^{4,\a}_{\de_\infty}(X_{p,q})}+\max\left\{\|\p\|_{C^{4,\a}_{\de}(A_{p})},\|\p\|_{C^{4,\a}_{\de}(A_{q})}\right\}+\|\p\|_{C^{4,\a}_{\de}(N_{p})}+\|\p\|_{C^{4,\a}_{\de}(N_{q})},$$ then we have 
\begin{align}
&\|\p\|_{C^{4,\a}_{\de,\de_\infty}(Bl_{p,q}X)}
\leq \|\cp\|_{C^{4,\a}_{\de,\de_\infty}(Bl_{p}X)}+\|\bcp\|_{C^{4,\a}_{\de,\de_\infty}(Bl_{q}X)} \\
\\
&\quad \leq K_{p}\|L_{\o_\e}(\cp)\|_{C^{0,\a}_{\de-4,\de_\infty-4}(Bl_p X)}+ K_1(\|\bcp\|_{C^{4,\a}(X_p)}+\|\bcp\|_{C^{4,\a}_{\de}(A_{p})})+\|\bcp\|_{C^{4,\a}_{\de}(N_{p})}\\
\\
&\quad \leq K_{p}\left(\|L_{\o}(\cp)\|_{C^{0,\a}_{\de_\infty-4}(X_{p,q})}+\|L_{\o}(\cp)\|_{C^{0,\a}_{\de_\infty-4}(A_q)}+\|L_{\o_{\e}}(\cp)\|_{C^{0,\a}_{\de-4}(A_p)}+\|L_{\e^2\eta_{BS}}(\cp)\|_{C^{0,\a}_{\de-4}(N_p)}\right.\\
&\quad\quad\quad \quad\left.+K_1\|L_{\o}(\bcp)\|_{C^{0,\a}(X_{p})}+K_1\|L_{\o_{\e}}(\bcp)\|_{C^{0,\a}_{\de-4}(A_p)}+\|L_{\e^2\eta_{BS}}(\bcp)\|_{C^{0,\a}_{\de-4}(N_p)}\right)\\
\\
&\quad \leq K_{p}\|L_{\o_\e}(\p)\|_{C^{0,\a}_{\de-4,\de_\infty-4}(Bl_{p,q} X)}+K_p\!\left(\|L_{\o}(\cp)\|_{C^{0,\a}_{\de_\infty-4}(A_q)}\!+\!K_1(\|L_{\o}(\bcp)\|_{C^{0,\a}(X_{p})}\!+\!\|L_{\o_{\e}}(\bcp)\|_{C^{0,\a}_{\de-4}(A_p)})\right).
\end{align}

\noindent Moreover, as $L$ is a linear operator at each metric, we get that there exists a constant $K_2$, depending on $p$, such that
\begin{equation*}
\begin{split}
\|L_{\o}(\cp)\|_{C^{0,\a}_{\de_\infty-4}(A_q)}&\leq K_2\|\cp\|_{C^{4,\a}_{\de_\infty}(A_q)}\\
\|L_{\o}(\bcp)\|_{C^{0,\a}(X_{p})}&\leq K_2\|\bcp\|_{C^{4,\a}(X_{p})}\\
\|L_{\o_{\e}}(\bcp)\|_{C^{0,\a}_{\de-4}(A_p)}&\leq K_2\|\bcp\|_{C^{4,\a}_{\de}(A_p)}
\end{split}
\end{equation*}

\noindent Also, we may assume that there exist constants $K_3$ and $K_4$ depending only on $p$ as $q$ goes to infinity, such that
$$\|\bcp\|_{C^{4,\a}(X_{p})}\leq K_3\|\bcp\|_{C^{4,\a}(X_{q})},\quad  \|\bcp\|_{C^{4,\a}_{\de}(A_p)}\leq K_4\|\bcp\|_{C^{4,\a}_{\de}(A_q)}.$$

\noindent Thus, there exists a constant $K$ which depends on $p$ as $q$ goes to infinity such that
\begin{equation*}
\|\p\|_{C^{4,\a}_{\de,\de_\infty}(Bl_{p,q}X)}\leq K_{p}\|L_{\o_\e}(\p)\|_{C^{0,\a}_{\de-4,\de_\infty-4}(Bl_{p,q} X)}+K_pK \left(\|\cp\|_{C^{4,\a}_{\de_\infty}(A_q)}+\|\bcp\|_{C^{4,\a}(X_{q})}+\|\bcp\|_{C^{4,\a}_{\de}(A_q)}\right).
\end{equation*}

\noindent Moreover, properties on the weighted norms imply
 $$\|\cp\|_{C^{4,\a}_{\de_\infty}(A_q)}\leq r(q)^{\de_\infty}\|\cp\|_{C^{4,\a}_{\de,\de_\infty}(Bl_q(X))},$$ 
$$\|\bcp\|_{C^{4,\a}(X_{q})}+\|\bcp\|_{C^{4,\a}_{\de}(A_q)}\leq r(q)^{\de_\infty}\|\bcp\|_{C^{4,\a}_{\de,\de_\infty}(Bl_q(X))},$$
where $r(q)$ is the distance between $q$ and a compact subset $K\subset X$ containing $p$.

\noindent Since $r(q)$ goes to infinity as $q$ goes to infinity and $\de_\infty<0$, we get
$$\|\p\|_{C^{4,\a}_{\de,\de_\infty}(Bl_{p,q}X)}\leq K_{p}\|L_{\o_\e}(\p)\|_{C^{0,\a}_{\de-4,\de_\infty-4}(Bl_{p,q} X)},\quad \mbox{as} \;q\; \mbox{goes to infinity}.$$

\end{proof}

\begin{proof}[Proof of Theorem \ref{epsilon}] Following Remark \ref{fixpointKC} and {\rm\bfseries{(i)}}, {\rm\bfseries{(ii)}}, if we choose $c_1=1/2K_{p,q}$ then
	$$\e_0(p,q):=\left(\frac{1}{2K_{p,q}\sqrt{C_{p,q}}}\right)^{\frac{m}{m-1}}$$
	and for every $\e<\e_0(p,q)$ we have
	$$K_{p,q}C_{p,q}r_\e^2<\frac{1}{2}c_1.$$
	
\noindent	Using Lemma \ref{C} and Lemma \ref{K}, we get, for $q$ sufficintely far away from $p$
	$$\e_0(p,q):=\left(\frac{1}{2K_{p,q}\sqrt{C_{p,q}}}\right)^{\frac{m}{m-1}}\geq\left(\frac{1}{2K_{p,q}\sqrt{C_{p}}}\right)^{\frac{m}{m-1}}\longrightarrow \left(\frac{1}{2K_{p}\sqrt{C_{p}}}\right)^{\frac{m}{m-1}}.$$
	
\noindent	Therefore, $\e_0(p,q)$ goes to $\e_0\geq \e_0(p).$
\end{proof}

\section{ADM mass under blow ups, proof of the main results}

\noindent We now analyze the effects of the previous results on the ADM mass of a scalar flat ALE space.
Thanks to the formula due to Hein-LeBrun (\cite{HL})  (Theorem \ref{admeq}), this is easily achieved by the following

\begin{Proposition}
\label{admblow}
Let $m(X,h)$ and $m(Bl_p X,h_\e)$ be the ADM mass of $(X,\Gamma, h, \omega)$ and $Bl_pX$ respectively, where $h_\e$ is the scalar flat metric produced out of $h$ by Theorem \ref{cscKALE}. Then,
$$m(Bl_p X,h_\e)=m(X,h) + \frac{(m-1)\e^{2(m-1)}}{(2m-1)\pi^{m-1}}.$$
\end{Proposition}

\begin{proof}
Since $\o_{\e}\in \pi^*[\o]-\e^2PD[E_p]$ and we have normalized the Burns-Simanca metric in a way that $[\eta_{BS}]=- PD[E_p]$ then
$$[\o_{\e}]=[\o]+\e^2[\eta], \quad c_1(Bl_pX)=\pi^*(c_1(X))-(m-1)PD[E_p].$$

\noindent By Theorem \ref{admeq} we have
\begin{eqnarray*}
	m(Bl_pX,h_\e)&=&-\frac{\langle \clubsuit(c_1(Bl_pX)),[\o_\e]^{m-1}\rangle}{(2m-1)\pi^{m-1}}\\
	&=&-\frac{\langle \clubsuit(c_1(X)-(m-1)PD[E_p]),([\o]-\e^2PD[E_p])\rangle}{(2m-1)\pi^{m-1}}\\
	&=&m(X,h) + \frac{(m-1)\e^{2(m-1)}}{(2m-1)\pi^{m-1}}.
\end{eqnarray*}
\end{proof}

\begin{proof}[Proof of Theorem  \ref{mainth1}:]
\noindent Proposition \ref{admblow} shows that each time we blow up a point $p$ we can increase the ADM mass of any amount up to $$\frac{(m-1)\e_0(p)^{2(m-1)}}{(2m-1)\pi^{m-1}}.$$ 

\noindent Theorem \ref{epsilon} shows that iterating the process taking a second point sufficiently far away from the first one, we can increase again the ADM mass by the same amount (at least). Of course, blowing up points does not change the structure groups at infinity of the original ALE space. Given that the process can be iterated infinitely many times, the theorem is proved.
\end{proof}

\end{document}